\documentclass[12pt]{amsart}
\usepackage{amsmath}
\usepackage{amsfonts,textcomp,longtable,pstricks,amscd}

\newtheorem*{thm*}{Theorem}

\newtheorem*{lemma*}{Lemma}

\newtheorem*{prop*}{Proposition}
\newtheorem{prop}[subsection]{Proposition }

\newtheorem*{corollary*}{Corollary}

\theoremstyle{definition}

\newtheorem*{rem*}{Remark}

\newtheorem{ex}[subsubsection]{Example }

\numberwithin{equation}{section}

\def\PP{{\mathbb P}}

\def\CC{{\mathbb C}}

\def\ZZ{{\mathbb Z}}

\def\AA{{\mathbb A}}
\def\FF{{\mathbb F}}
\def\fq{{\mathbb{F}_q}}
\def\kk{{\bar{k}}}
\def\aa{{\bar{a}}}
\def\CC{{\overline{C}}}
\def\XX{{\overline{X}}}

\def\et{{{\acute e}t}}
\def \bra#1\ket {\mathop{\vphantom{#1}\left<\smash{#1}\right>}\nolimits}
\def\sep{{\mathrm sep}}
\def\ge{{\mathrm{geom}}}
\def\tame{{\mathrm{tame}}}
\def\wild{{\mathrm{wild}}}

 \DeclareMathOperator{\Hom}{Hom}
\DeclareMathOperator{\Aut}{Aut}

\DeclareMathOperator{\Gal}{Gal} \DeclareMathOperator{\Spec}{Spec}

\DeclareMathOperator{\GL}{GL}

\def\V{\mathcal{V}}
\def\E{\mathcal{E}}

\def\plim{\mathop{{\lim\limits_{\longleftarrow}}}\nolimits}

\renewcommand \phi {\varphi}

\begin{document}
\author{Sergey Rybakov}
\thanks{The research was carried out at the IITP RAS at the expense of the Russian Foundation for Sciences (project $N^{\underline{o}}$ 14-50-00150).}
\address{Institute for information transmission problems of the Russian Academy of Sciences}
%\address{Poncelet laboratory (UMI 2615 of CNRS and Independent University of Moscow)}
%\address{ AG Laboratory, HSE, 7 Vavilova str., Moscow, Russia, 117312 }

\email{rybakov.sergey@gmail.com}%
\title{Families of algebraic varieties and towers of algebraic curves over finite fields}
\date{}


\bigskip                                         

\begin{abstract}
We introduce a new construction of towers of algebraic curves over finite fields and provide a simple example of an optimal tower.
\end{abstract}

\maketitle
\section{Introduction}
In this paper we use the following notation. Let $k$ be a field of positive characteristic $p$.
We fix an algebraic closure $\kk$ of $k$ and denote by $\overline{Y}$ the base change $$Y\times_{\Spec k}\Spec\kk$$ of an algebraic variety $Y$ over $k$.
We also fix a prime number $\ell\neq p$. 

A \emph{tower of algebraic curves} over $k$ is an infinite sequence 
$$\dots C_n\to C_{n-1}\to\dots\to C_0$$ of smooth projective and geometrically connected curves and finite morphisms. 
We assume that the genus $g(C_n)$ is unbounded. If $k=\fq$ is a finite field,
the number of points $|C_n(k)|$ on the curve $C_n$ is defined, and
the limit $$\beta(C_\bullet)=\lim_{n\to\infty}\frac{|C_n(k)|}{g(C_n)}$$ exists. Moreover, by the Drinfeld-Vl\v{a}du\c{t} theorem~\cite[3.2.1]{TsVN}, $\beta(C_\bullet)\leq\sqrt{q}-1$. 
The tower is called \emph{optimal}, if $\beta(C_\bullet)=\sqrt{q}-1$. 
It is known that if $q$ is a square, examples of optimal towers over $\FF_q$ can be constructed as modular towers (see ~\cite{TsVZ}, and~\cite{I}), 
but in other cases, all known towers are far from being optimal. 

In this paper, we introduce a new construction of towers of algebraic curves over finite fields. 
We begin with a family $f:X\to C$ of algebraic varieties over a projective curve $C$.
Assume that $f$ is smooth over an open subset $U\subset C$. 
The $i$-th derived \'etale direct image of the constant sheaf $\ZZ/\ell^n \ZZ$ corresponds to a local system $\V_n$ on $U$. 
There is a fiberwise \emph{projectivisation} $P_n(\V_n)$ of this local system (see section~\ref{s2}), which is an \'etale scheme $U_n$ over $U$. 
Clearly, $U_n$ is a regular scheme of dimension one, and there exists a regular compactification $C_n$ of $U_n$. 
If $a\in C(\kk)$ is a closed point such that the fiber $X_a$ of $X$ over $a$ is smooth, then the set of $\kk$--points on the fiber of $C_n$ over $a$ is
canonically isomorphic to the quotient $$H_\et^i(\XX_a,\ZZ/\ell^n \ZZ)^*/(\ZZ/\ell^n \ZZ)^*.$$ In particular, this isomorphism respects Frobenius action if $a\in C(k)$.
We prove that under some strong technical conditions on the family $X$ the scheme $C_n$ is a geometrically irreducible algebraic curve. 
As an example, we construct \emph{the Legendre tower} starting from the Legendre family of elliptic curves over $\FF_{p^2}$, and prove that it is an optimal tower. 

Our construction is related to modular towers as follows. Recall that points on the curve $X_0(\ell^n)$ correspond to
isomorphism classes of elliptic curves $E$ with a cyclic subgroup in $E(\kk)$ of order $\ell^n$. 
In fact, this curve is defined over $\FF_{p^2}$ (and even over $\FF_p$), and the family $X_0(\ell^\bullet)$ is an optimal tower over $\FF_{p^2}$. 
An elliptic curve $E$ over $\kk$ defines a point on $X_0(1)$, and the fiber of $X_0(\ell^n)$ over this point is isomorphic to the projectivisation 
$P_n(E[\ell^n](\kk))$ of the $\ell^n$--torsion. 
Let $\E\to\PP^1$ be a family of elliptic curves such that its $j$--invariant is a map of degree $1$. 
We expect that, for an appropriate family $\E$, our construction gives the modular tower $X_0(\ell^\bullet)$, 
and the Legendre tower is a base change of $X_0(\ell^\bullet)$.

One of the conditions on the family is connected with supersingularity of the smooth fibers. We say that a smooth variety $Y$ over $k$ is 
\emph{strongly supersingular in degree $i$} if the Frobenius action on $H^i_\et(\overline{Y},\ZZ_\ell)$ is multiplication by $q^{i/2}$ or $-q^{i/2}$.  
Clearly, if $Y$ is strongly supersingular, then it is supersingular in the usual sense. 
Moreover, we show that if the fiber $X_a$ over $a\in C(k)$ is strongly supersingular, then the fiber of $C_n$ over $a$ is split for all $n$.   

We hope that our construction will help to find towers with good asymptotic properties over $\FF_p$. 
Note that the eigenvalue $\pm p^{i/2}$ of the Frobenius action on $H^i_\et(\overline{Y},\ZZ_\ell)$ is integral only if $i$ is even. 
In particular, if $Y$ is strongly supersingular in degree $i$ over $\FF_p$, then
$i$ is even. This observation forces us to study families of algebraic varieties of dimension at least two.

We discuss here the case $i=2$ and $\dim X=3$, i.e., $f:X\to C$ is a family of algebraic surfaces.                                                          
One of the most important restrictions on the family comes from the monodromy representation. 
Namely, the monodromy group has to be infinite. 
It follows that families of rational surfaces do not lead to interesting towers, 
because the image of the fundamental group in the automorphism group of the Neron--Severi group of a surface is finite. 
On the other hand, the example of Legendre family shows that towers assigned to families of abelian varieties 
are closely related to known modular towers, and these towers are non-optimal over $\FF_p$.
Therefore it is natural to consider algebraic surfaces that are far from being rational or abelian.
This will be the object of our future research.

\section{The construction}\label{s2}
\subsection{} Denote the ring $\ZZ/\ell^n\ZZ$ by $\Lambda_n$.
Let $V$ be a finitely generated $\Lambda_n$--module. The set $$V^*=\{v\in V|\ell^{n-1}v\neq 0\}$$ has the natural action of the group of invertible elements $\Lambda^*_n$.
We say that the set $$P_n(V)=V^*/\Lambda_n^*$$ is \emph{the projectivisation of $V$}. 
For example, $P_1$ is the usual projectivisation of an $\FF_\ell$-vector space. Recall that the cardinality of the set $P_1(\Lambda_1^{b})$ is equal to $$c_\ell(b)=\frac{\ell^{b}-1}{\ell-1}.$$

\begin{lemma*}
Suppose that $V_n$ is a $\Lambda_n$--module, and $V_{n-1}$ is a $\Lambda_{n-1}$--module. The group $V_{n-1}$ is a $\Lambda_n$--module in an obvious way. 
Let $\phi:V_n\to V_{n-1}$ be a homomorphism such that $\ell\ker\phi=0$.
Then there is an induced map $$P(\phi): P_n(V_n)\to P_{n-1}(V_{n-1}).$$ 
\end{lemma*}
\begin{proof}
Let $v\in V_n^*$. Since $\ell\ker\phi=0$, we have $\ell^{n-2}\phi(v)\neq 0$. This gives a morphism $\phi^*:V_n^*\to V_{n-1}^*$ which is compatible with the actions of multiplicative groups. 
\end{proof}

For example, the natural projection $\Lambda_n^b\to \Lambda_{n-1}^b$ induces a morphism $$P_n(\Lambda_n^b)\to P_{n-1}(\Lambda_{n-1}^b)$$ of degree $\ell^{b-1}$.
It follows that the cardinality of  $P_n(\Lambda_n^b)$ is $c_\ell(b)\ell^{(b-1)(n-1)}$.

\subsection{The general case.}\label{constr} Let $U$ be a smooth curve over $k$, and let $\V_\bullet$ be a projective system of locally constant sheaves 
$\V_n$ of free $\Lambda_n$--modules on the small \'etale site of $U$. We assume that the rank $b$ of $\V_n$ does not depend on $n$. 
By~\cite[V.1.1]{Milne}, the functor $\V_n^*$ given by 
$$T\mapsto \V_n(T)^*$$ is representable by a finite \'etale scheme over $U$ with a free action of the group $\Lambda_n^*$. 
Thus, the quotient by this action represents the sheaf $\Tilde{U}_n=P_n(\V_n)$. 

\begin{thm*} 
The sheaf $\Tilde{U}_n=P_n(\V_n))$ is representable by a finite \'etale scheme $U_n$ over $U$. 
\end{thm*}

Let $C_0$ be a smooth projectivisation of $U$, and let $C_n$ be a regular projectivisation of $U_n$. 
Note that every connected component $Y$ of $C_n$ is a smooth algebraic curve over the algebraic closure of $k$ in $k(Y)$.

\begin{lemma*}%\label{fiber}
The fiber of $C_n$ over $\aa\in U(\kk)$ is naturally isomorphic to $P_n((\V_n)_\aa)$.
\end{lemma*}
\begin{proof}
By~\cite[V.1.7(c)]{Milne}, the scheme $U_n$ represents the sheaf $g^*\V_n$, where $g$ is the tautological morphism from the 
big \'etale site to the small \'etale site of $U$. Untwisting the definitions we get the lemma.
\end{proof}

\begin{corollary*}
The degree of $C_1$ over $C_0$ is equal to $c_\ell(b)$.
If $n>1$, then the natural projection $\V_n\to \V_{n-1}$ induces a morphism $C_n\to C_{n-1}$ of degree $\ell^{b-1}$.
\end{corollary*}

\subsection{The case of a generically smooth family of algebraic varieties.}\label{family}
Let $f:X\to C$ be a family of algebraic varieties. We assume that $f$ is smooth over $U\subset C$, and
denote by $f_U:X_U\to U$ the corresponding smooth family.
By~\cite[VI.4.2]{Milne}, we have a locally constant finite sheaf of $\Lambda_n$--modules on $U$: $$\V_n=R_\et^if_{U*}(\Lambda_n),$$
where the direct image is taken with respect to the small \'etale site of $U$. 

A point $\Bar a\in U(\kk)$ corresponds to a morphism $\Spec\kk\to U$.
By~\cite[VI.2.5]{Milne}, the stalk $(\V_n)_{\Bar a}$ of $\V_n$ is isomorphic to $H^i_\et(\XX_{\Bar a},\Lambda_n)$. 
In what follows, we assume that the module $H^i_\et(\XX_{\Bar a},\ZZ_\ell)$ is free, and denote the corresponding Betti number by $b$. 

As in Section~\ref{constr}, we construct a tower of schemes $C_n$ starting from $\V_n$. 
We are going to discuss under what conditions on the family $f$ the tower $C_n$ is a tower of (geometrically connected) algebraic curves.

\subsection{The fundamental group of $U$.}\label{Pi1} Recall some basic properties of the fundamental group of a $k$-scheme $Z$~\cite[I.5]{Milne}.
 Let $K$ be a separably closed field. 
Suppose that there is a morphism $\Bar a:\Spec K\to Z$. 
The fundamental group $\pi_1(Z,\Bar a)$ is the automorphism group of the functor from finite \'etale $Z$-schemes to sets $$T\mapsto \Hom_Z(\Spec K,T).$$
This is a profinite group, and for any other point $\Bar a'$ there exists an isomorphism of profinite groups $i_{\Bar{aa}'}:\pi_1(Z,\Bar a)\to \pi_1(Z,\Bar a')$.
This isomorphism is unique up to inner automorphism of $\pi_1(Z,\Bar a)$.

\begin{ex}\label{I52}
By~\cite[I.5.2(a)]{Milne}, if $Z=\Spec L$ is the spectrum of a field, then it has only one point $x$, and  $\pi_1(Z,\Bar x)$ is the Galois group $\Gal(L)$ of a separable closure $L^\sep$ of $L$.
The functors $L'\mapsto \Hom_L (L',L^\sep)$ and $$S\mapsto (\prod_{s\in S}L^\sep)^{\Gal(L)}$$ establish an equivalence between
the category of \'etale extensions of $\Spec L$ and the category of finite $\Gal(L)$-sets. Moreover, finite field extensions of $L$ correspond to transitive actions of $\Gal(L)$ on finite sets.
\end{ex}

A morphism of $k$-schemes $\phi:Z\to Y$ induces a homomorphism of fundamental groups $$\pi_1(Z, \bar a)\to\pi_1(Y,\phi(\Bar a)).$$
For example, let $a$ be a closed point on $U$; in other words we have a morphism $a:\Spec k(a)\to U$.
Choose an embedding $k(a)\to k(a)^\sep$ to a separable closure of $k(a)$. This embedding defines a morphism  $\bar a:\Spec k(a)^\sep\to U$.
It follows that there is a homomorphism $\pi_1(U,\Bar a)\to \Gal(k)$. If $k$ is a finite field, then $\Gal(k)\cong \widehat\ZZ$ is topologically generated by Frobenius.
In this case we have an exact sequence of profinite groups $$1\to G^\ge\to \pi_1(U,\Bar a)\to\widehat\ZZ\to 1,$$ 
where $G^\ge=\pi_1(\overline{U},\Bar a)$ is the geometric part of the fundamental group.

Let $\eta=\Spec k(C)$ be the generic point of $C$. Then there is a surjective homomorphism $\Gal(k(U))\to\pi_1(U,\Bar\eta)$~\cite[I.5.2(b)]{Milne}.

\begin{lemma*}
If the action of $\pi_1(U,\Bar a)$ is transitive on $P_n((\V_n)_{\Bar a})$, then $C_n$ is a connected scheme.
\end{lemma*}
\begin{proof} The group $\pi_1(U,\aa)$ is isomorphic to the quotient of the Galois group $\Gal(k(U))$. 
By~\ref{I52}, finite field extensions of $k(U)$ correspond to transitive actions of $\Gal(k(U))$ on finite sets. 
Since the action on $P_n((\V_n)_{\aa})$ is transitive, the generic point of $C_n$ is the spectrum of a field; thus $C_n$ is connected.
\end{proof}

\begin{corollary*} In the notation of the section~\ref{family}, if the natural action of $\pi_1(U,\aa)$ 
is transitive on $P_n(H^i_\et(\overline{X}_\aa,\Lambda_n))$, then $C_n$ is a connected scheme.
\end{corollary*}

\subsection{The Frobenius action.}\label{Faction} Assume that $k$ is finite. Let $a\in U(k)$. 
The fiber $(\V_n)_\aa$ is representable by an abelian object of the small \'etale site of $\Spec k$. 
By Example~\ref{I52}, this is an abelian group endowed with a continous action of $\Gal(\kk/k)\cong \Hat\ZZ$.
This group is topologically generated by the Frobenius automorphism $F$, thus $(\V_n)_\aa$ is simply a finite abelian group 
plus an action of Frobenius.
%There is a natural Frobenius action on the fiber $(\V_n)_\aa$~\cite[VI.13]{Milne}. 
The following assumption plays a crucial role in this paper:

$(F_{a,n})$ The Frobenius action on $(\V_n)_\aa$ is multiplication by a constant.

In our main example from Section~\ref{family}, we have $\V_n=R_\et^if_{U*}(\Lambda_n)$ and $(\V_n)_\aa\cong H^i_\et(\overline{X}_\aa,\Lambda_n)$.
%Recall that the fiber $X_a$ of the morphism $f:X\to C$ is a smooth variety over $k$.
By the Weil conjectures proved by Deligne, $(F_{a,n})$ for all $n$ is equivalent to the strong supersingularity of $X_a$ in degree $i$.

\begin{lemma*} Assume that there exists a point $a\in U(k)$ such that condition $(F_{a,n})$ is satisfied. Then the following assertions hold.
\begin{enumerate}
\item The fiber of the morphism $C_n\to C_0$ over $a$ is split.
\item The scheme $C_n$ is a smooth curve over $k$.
\item If $C_n$ is connected, then $C_n$ is a geometrically irreducible curve over $k$.
\end{enumerate}
\end{lemma*}
\begin{proof}
By Lemma~\ref{constr}, the fiber and $P_n((\V_n)_\aa)$ are isomorphic as sets with Frobenius action. Condition $(F_{a,n})$ means that,
 for any point $x\in C_n(\kk)$ of the fiber, we have $F(x)=x$, where $F$ denotes Frobenius action. Thus, $x\in C_n(k)$ is a closed point, and the fiber is split. Moreover, the algebraic closure of $k$ in $k(C_n)$ is equal to $k$; in other words, $C_n$ is a $k$--variety. This proves part $(2)$.
If $C_n$ is a connected $k$--scheme, then $k(C_n)$ is a field such that $\kk\cap k(C_n)=k$ as subfields of an algebraic closure of $k(C_n)$.
We see that the composit of $\kk$ and $k(C_n)$ is isomorphic to $\kk(\CC_n)$, i.e., $\kk(\CC_n)$ is a field, and $C_n$ is geometrically irreducible.
\end{proof}

\begin{corollary*}
In the notation of the section~\ref{family}, if $C_n$ is connected, and there exists $a\in U(k)$ 
such that $X_a$ is strongly supersingular in degree $i$, then $C_n$ is a smooth, geometrically irreducible curve over $k$, and the fiber of the morphism $C_n\to C_0$ over $a$ is split.
\end{corollary*}

\subsection{The monodromy representation.}\label{Monodromy} Let $a\in \overline U(\kk)$ be a closed point. 
The action of $G^\ge$ on $(\V_n)_{\Bar a}$ induces homomorphisms
$$\rho_n:G^\ge\to \GL_{b}(\ZZ/\ell^n\ZZ),$$
$$\rho:G^\ge\to \GL_{b}(\ZZ_\ell).$$
By definition, the induced action on $P_n((\V_n)_{\aa})$ is isomorphic to the action on the fiber of the morphism $\CC_n\to \CC_0$ over $a$.

We can describe the fundamental group locally. Let $A=\kk[[t]]$ be the ring of formal power series, and let $K=\kk((t))$ be its fraction field. The scheme  $\Spec A$ has the only closed point $\alpha$, and the generic point $\eta:\Spec K\to\Spec A$ given by the inclusion $A\to K$. 
The fundamental group $G_\eta=\pi_1(\Spec K,\overline\eta)$ is isomorphic to $\Gal(K^\sep/K)$. Thus there is an exact sequence:
$$1\to G_\eta^\wild\to G_\eta\to G_\eta^\tame\to 1,$$ where $G_\eta^\wild$ is a $p$-group, and $$G_\eta^\tame\cong\prod_{\ell\neq p}\ZZ_\ell$$ is the Galois group 
of the extension of $K$ generated by $\sqrt[n]{t}$ for all $n$ coprime to $p$. The group $G_\eta^\tame$ has a topological generator $\xi$ such that the action on $\sqrt[n]{t}$ is the multiplication by a primitive $n$-th root of unity.

For any point $y\in C(\kk)$ there is a morphism $\phi_y:\Spec A\to C$ such that $\phi_y(\alpha)=y$, and there is a commutative diagram
$$\begin{CD}
\Spec K@>\phi_\eta>>U\\
@V\eta VV @V  VV\\
\Spec A@> \phi_y>> C
\end{CD}$$

 Let $\rho_y:G_\eta\to \GL_{b}(\ZZ_\ell)$ be the composition of $\rho$ with the induced homomorphism $\Tilde\phi_\eta:G_\eta\to G^\ge$. 
The representation $\rho$ is called \emph{unramified} in $y$ if $\rho_y(G_\eta)=1$, and $\rho$ is called \emph{tame} in $y$ if $\rho_y(G^\wild)=1$.
We denote by $G_y$ the image of $G_\eta$ in $G^\ge$.
Note that $\Tilde\phi_\eta$ (and $G_\eta$ as well) depends on the choice of the point $\aa\in U(\kk)$; in the tame case this is not important, since $G_y$ is commutative.
% if we choose another point, then $\Tilde\phi_\eta$ 
Therefore in the tame case $\rho_y$ is uniquely determined by $M=\rho_y(\xi)$. We say that $M$ is \emph{the monodromy operator}. 
\begin{lemma*}
If $y\in U(\kk)$, then $\rho$ is unramified in $y$.
\end{lemma*}
\begin{proof}
The morphism $\phi_\eta$ factors through $\Spec A$; thus $\rho_y$ factors through $\pi_1(\Spec A,\alpha)$. By~\cite[I.5.2(b)]{Milne}, the last group is trivial. 
\end{proof}

%\begin{corollary*}
%We use notation of the section~\ref{family}.
%If the fiber $X_y$ is smooth, then $\rho$ is unramified in $y$.
%\end{corollary*}

\subsection{}\label{Moperator}
Let $y\in C(\kk)$ be a point such that the fiber $X_y$ is not smooth. Assume that $\rho$ is tame in $y$. 
Fix a point $\Bar a\in U(\kk)$. We are going to compute ramification of $\overline{C}_n$ over $y$ in terms of the monodromy operator $M$. 

\begin{prop*}
There is a bijection between orbits of the action of $G_y$ on $P_n((\V_n)_{\Bar a})$ and points $x\in C_n(\kk)$ over $y$. 
The ramification index of $x$ is equal to the cardinality of the corresponding orbit.
\end{prop*}
\begin{proof}
An orbit $S$ of the action of $G_y$ on $P_n((\V_n)_{\Bar a})$ corresponds to a finite extension $K_S$ of $K$. 
Moreover, since the representation is tame, $K_S\cong K[\sqrt[r]{t}]$, where $r=|S|$. Let $A_S$ be the normalization of $A$ in $K_S$. Then there is a morphism $\phi_S:\Spec A_S\to C_n$ such that the diagram 
$$\begin{CD}
\Spec A_S@>\phi_S>>C_n\\
@V VV @V  VV\\
\Spec A@> \phi_a>> C_0
\end{CD}$$
is commutative. It follows that the unique point $\alpha_S\in\Spec A_S$ over $\alpha\in\Spec A$ goes to a point $x$ of $C_n$ over $y$ with ramification index $r$. 
\end{proof}

\section{Families of elliptic curves} 
\subsection{}In this section we assume that $k=\fq$ is a finite field of odd characteristic $p$.
Let $E$ be an elliptic curve over $k$. Denote by $E_m$ the kernel of the multiplication by $\ell^m$ in $E(\kk)$. 
The $\ell$-th Tate module of $A$ is defined by the formula: $$T_\ell(E) = \plim E_m.$$ 
The module $T_\ell(E)$ is a free $\ZZ_\ell$-module of rank $2$.
The Frobenius endomorphism $F$ of $E$ acts on the Tate module by a semisimple linear transformation, which we also denote by $F$. 
The characteristic polynomial
$$
f_E(t) = \det(t-F|T_\ell(E))
$$
is called \emph{the Weil polynomial of $E$}. It is a monic polynomial of degree $2$ with rational integer coefficients independent of the choice of the prime $\ell$. 

The Tate module $T_\ell(E)$ is isomorphic to $$\Hom_{\ZZ_\ell}(H^1_\et(\overline{E},\ZZ_\ell),\ZZ_\ell).$$
Therefore, $E$ is strongly supersingular in degree $1$ if and only if 
$$f_{E}(t)=(t\pm\sqrt{q})^2.$$ In particular, $E$ is supersingular, and $q$ is a square. 

\begin{prop} \label{prop1}
Assume that $q=p^2$, and an elliptic curve $E$ is supersingular. Then the Weil polynomial of $E$ is equal to $f(t)=(t\pm p)^2$ if one of the following conditions is satisfied:
\begin{enumerate}
	\item $p\equiv -1\mod 12$;
	\item $E$ is defined over $\FF_p$, and $p\neq 3$; 
	\item $E[2](k)\cong (\ZZ/2\ZZ)^2$, in other words, $2$-torsion is defined over $k$.
	\end{enumerate}
\end{prop}
\begin{proof}
We use the Deuring-Waterhouse classification of Weil polynomials~\cite[Theorem 4.1]{Wa}. 
Statement $(1)$ is easy, and $(2)$ follows from the observation that the Weil polynomial of a supersingular elliptic curve over $\FF_p$ is 
equal to $t^2+p$. Conditions of $(3)$ imply that the Weil polynomial $f(t)$ is congruent to $t^2+1$ modulo $2$. 
Thus, either $f(t)=(t\pm p)^2$, or $f(t)=t^2+p^2$. But in the latter case, the number of points on the curve $f(1)=1+p^2\equiv 2\bmod 4$ is not divisible by $4$. 
\end{proof}

\subsection{Example: Legendre family}
In this section we prove that for large $\ell$ the Legendre family gives an optimal tower of algebraic curves over $k=\FF_{p^2}$, where $p>3$.

Let $C=\PP^1$, and $X\to C$ be the desingularisation of the Legendre family given in affine coordinates by the equation
$$y^2=x(x-1)(x-a),$$ where $a$ is a coordinate on $\AA^1\subset \PP^1$. This family has two degenerate fibers of Kodaira type $I_2$ 
(see~\cite[IV.9]{S}) over $0$, $1$, and one fiber of type $I_2^*$ over infinity.
If $\ell$ is large, the monodromy representation is tame for all points $a\in C_0(\kk)$, and the monodromy matrix for such 
degenerations is conjugate to $\pm M$, where
$$M=\left (\begin{array}{c c}
1 & 2\\
0 & 1\\	
\end{array}\right).$$

We call the corresponding tower $C_n$ \emph{the Legendre tower.}
The $j$-invariant is not constant for the Legendre family and, by~\cite[IV.3.2]{S68}, 
for large $\ell$ the natural morphism $\pi_1(U,\Bar a)\to\Aut(T_\ell(X_a))$ is surjective.
By Corollary~\ref{Pi1}, the schemes $C_n$ are connected. 

A smooth fiber over a point $a\in\AA^1(k)$ is supersingular if and only if $a$ is a root of the Hasse polynomial 
$$H(t)=\sum_{i=0}^{\frac{p-1}{2}}\left(\begin{array}{c}
\frac{p-1}{2} \\	
i \\
\end{array}
\right)^2t^i.$$
It is known that $H(t)$ is separable, and all roots of $H(t)$ belong to $k-\{0,1\}$. 
Moreover, the $2$--torsion of a smooth fiber over a point $a$ is the zero of the group law, and three points with coordinates $y=0$, and $x\in \{0,1,a\}$.
By Proposition~\ref{prop1}.(3), if $H(a)=0$, this fiber is a strongly supersingular elliptic curve. 
Since $H(t)$ is separable, there are exactly $\frac{p-1}{2}$ such fibers.
It follows form Corollary~\ref{Faction} that $C_n$ are geometrically irreducible curves over $k$. 
By Lemma~\ref{Faction}.(1), all fibers over roots of $H$ are split; therefore, $$|C_n(k)|\geq \frac{p-1}{2}(\ell+1)\ell^{n-1}.$$ 
Now, we compute the ramification of $C_n$ over $C_0$ using Proposition~\ref{Moperator}. 

\begin{lemma*} Assume $n$ is even. There are three types of orbits of the action of $G_a$ on $P_n((\ZZ/\ell^n\ZZ)^2)$, where $a\in\{0,1,\infty\}:$
\begin{enumerate}
\item one orbit of length $\ell^n$;
\item for each $i\in\{1,\dots, n/2-1\}$ there are $(\ell-1)\ell^{i-1}$ orbits of length $\ell^{n-2i}$;
\item $\ell^{n/2}$ orbits of length one. 
\end{enumerate}
\end{lemma*}
\begin{proof}
We asuume that the monodromy matrix is conjugate to $M$; the $-M$ case is similar.
Let $v_0, v_1$ be a basis of $(\ZZ/\ell^n\ZZ)^2$ such that $Mv_0=v_0$, and $Mv_1=v_1+2v_0$.
We denote the class of a vector $v\in(\ZZ/\ell^n\ZZ)^2$ in $P_n((\ZZ/\ell^n\ZZ)^2)$ by $[v]$.
The orbit of type $(1)$ is formed by classes $[v_1+\lambda v_0]$, where $\lambda\in(\ZZ/\ell^n\ZZ)$. 

Let $i\geq 1$, and let $\lambda\in(\ZZ/\ell^{n-i}\ZZ)^*$. Put $$v_{\lambda,i}=v_0+\ell^i\lambda v_1.$$
We claim that $[v_{\lambda_1,i}]$ and $[v_{\lambda_2,i}]$ are in one orbit if and only if $\lambda_1\equiv \lambda_2\mod \ell^i$.
If $n>2i$, we say that an orbit of type $(2)$ is formed by classes of vectors $$v_{\lambda,i}+ \ell^{2i}\mu v_1,$$ where $\mu\in  \ZZ/\ell^{n-2i}\ZZ$.
If $n\leq 2i$, then we define an orbit of type $(3)$ as the class of $v_{\lambda,i}$. Additionally, $[v_0]$ is also of type $(3)$.

To prove the claim choose a natural number $s$ such that
$$2\lambda^2s\equiv -1\mod \ell^{n-2i}.$$
%Since both $-2$ and $\lambda$ are invertible in $\ZZ/\ell^{n-i}\ZZ$ such an $s$ exists.
It is straightforward to check that $$M^s[v_{\lambda,i}]=[v_{\lambda+\ell^i,i}].$$ 
The lemma is proved.
\end{proof}

The Hurwitz formula gives
$$g(C_n)=-(\ell+1)\ell^{n-1}+\frac{3}{2}(\ell^n-1)+\frac{3}{2}\sum_{i=1}^{n/2-1}(\ell-1)\ell^{i-1}(\ell^{n-2i}-1)+1=\frac{1}{2}(\ell^{n-1}(\ell+1)-3(\ell+1)\ell^{n/2-1})+1.$$
We finally obtain
$$\beta(C_\bullet)\geq p-1.$$ It follows from the Drinfeld-Vl\v{a}du\c{t} bound that the Legendre tower is optimal.

\end{document}